\documentclass[11pt]{article}

\usepackage[T1]{fontenc}
\usepackage{lmodern}
\usepackage{microtype}
\usepackage[a4paper,margin=30mm]{geometry}
\usepackage{amsmath,amssymb,amsthm,mathtools}
\usepackage{hyperref, comment}

\hypersetup{
  pdftitle={Rank-One Convexity Implies Quasiconvexity for the Dacorogna–Marcellini Energy},
  pdfauthor={},
  hidelinks
}

\allowdisplaybreaks
\numberwithin{equation}{section}

\newtheorem{theorem}{Theorem}
\newtheorem{lemma}[theorem]{Lemma}

\newcommand{\R}{\mathbb R}
\newcommand{\M}{\R^{2\times2}}
\newcommand{\dd}{\mathrm d}
\newcommand{\cof}{\operatorname{cof}}
\newcommand{\diag}{\operatorname{diag}}

\title{Quasiconvexity for the Dacorogna–Marcellini Energy}
\author{Giuseppe Bruno\textsuperscript{1} and
Federico Pasqualotto\textsuperscript{2}\\[0.5em]
\small\textsuperscript{1}University of Bern and Google DeepMind\\
\small\textsuperscript{2}University of California San Diego and Google DeepMind}
\date{\today}

\begin{document}

\maketitle

\begin{abstract}
We prove that the planar Dacorogna--Marcellini energy
\(f_\gamma(A)=|A|^4-2\gamma|A|^2\det A\) is quasiconvex exactly when
it is rank-one convex, i.e. if and only if
 $|\gamma|\leq\frac{2}{\sqrt3}$.
The proof uses a monotonicity property of the energy functional along the componentwise heat flow. As a corollary of our method, we show that for homogeneous quartic polynomials on $2 \times 2$ matrices invariant by left and right rotation, quasiconvexity is equivalent to rank one convexity.
\end{abstract}

\section{Introduction}

Let \(\Omega\subset\R^n\) be a bounded domain and consider a
first-order variational integral
\[
 \mathcal I_f(u)=\int_\Omega f(Du(x))\,\dd x,
 \qquad
 u:\Omega\to\R^m,
\]
where \(f:\R^{m\times n}\to\R\) is continuous.  In the direct method
of the calculus of variations, weak lower semicontinuity of
\(\mathcal I_f\) allows one to pass to the weak limit of a minimizing
sequence and is therefore a fundamental property in existence results
\cite{Dacorogna2008}.
For scalar problems, ordinary convexity gives the relevant condition,
whereas for vector-valued maps the situation is more difficult.

Morrey introduced quasiconvexity in his study of this question
\cite{Morrey1952}.  An integrand
\(f:\R^{m\times n}\to\R\) is \emph{quasiconvex} if
\begin{equation}\label{eq:quasiconvexity}
 \int_U f(F+D\varphi)\,\dd x
 \geq |U|f(F)
\end{equation}
for every bounded Lipschitz domain \(U\subset\R^n\), every
\(F\in\R^{m\times n}\), and every
\(\varphi\in W^{1,\infty}_0(U,\R^m)\).   When \(m=1\) or \(n=1\),
this condition is equivalent to convexity, whereas for \(m,n\geq2\)
it is much harder to check.

Since the definition is global in the perturbation \(\varphi\),
quasiconvexity is difficult to verify directly.  Rank-one convexity
provides a simpler necessary condition: the integrand \(f\) is
\emph{rank-one convex} if
\[
 f\bigl((1-t)A+tB\bigr)
 \leq (1-t)f(A)+tf(B)
\]
whenever \(t\in[0,1]\), \(A,B\in\R^{m\times n}\), and
\(\operatorname{rank}(A-B)\leq1\).  For a twice differentiable
integrand, this is equivalent to the Legendre--Hadamard inequalities
\begin{equation}\label{eq:legendre-hadamard}
 D^2f(A)[a\otimes b,a\otimes b]\geq0,
 \qquad
 a\in\R^m,\quad b\in\R^n.
\end{equation}
Every quasiconvex integrand is rank-one convex, and the question of
whether the converse holds is known as Morrey's problem.

Polyconvexity provides a useful sufficient condition: an integrand is
\emph{polyconvex} if it is a convex function of all the minors of its
matrix argument.  In dimension \(2\times2\), this means that
\[
 f(A)=g(A,\det A)
\]
for some convex function \(g:\R^{2\times2}\times\R\to\R\), a
condition introduced by Ball in nonlinear elasticity
\cite{Ball1977}.  Together, these notions satisfy
\[
 \text{convex}
 \quad\Longrightarrow\quad
 \text{polyconvex}
 \quad\Longrightarrow\quad
 \text{quasiconvex}
 \quad\Longrightarrow\quad
 \text{rank-one convex}.
\]
Although none of the reverse implications holds in full generality,
the last one is the most relevant here.  {\v S}ver\'ak constructed a
rank-one-convex integrand on
\(\R^{3\times2}\) which is not quasiconvex \cite{Sverak1992}.  This
settled the question in higher target dimension, but the construction
does not descend to \(\R^{2\times2}\), which became the basic testing
ground for the gap between rank-one convexity and quasiconvexity.

The planar problem has also been studied by imposing additional
structure on the matrices or on the integrand.  M\"uller proved that
rank-one convexity implies quasiconvexity on diagonal matrices
\cite{Muller1999}, and Harris, Kirchheim, and Lin extended this result
to upper-triangular \(2\times2\) matrices
\cite{HarrisKirchheimLin2018}.  Sz\'ekelyhidi described the geometry of
rank-one-convex hulls in \(\R^{2\times2}\) and of quasiconvex hulls in
the symmetric subspace \cite{Szekelyhidi2005,Szekelyhidi2006}.
Faraco and Sz\'ekelyhidi proved a localization theorem for quasiconvex
hulls in \(\R^{2\times2}\), while Kirchheim and Sz\'ekelyhidi obtained
a characterization of incompatible sets of planar gradients
\cite{FaracoSzekelyhidi2008,KirchheimSzekelyhidi2008}.  For
rotationally invariant integrands, Dacorogna and Koshigoe showed that convexity and
polyconvexity can be tested on diagonal matrices, although the same
reduction is not valid for rank-one convexity and quasiconvexity
\cite{DacorognaKoshigoe1993}.  A stronger conclusion holds in the
isochoric class from nonlinear elasticity, where Martin, Ghiba, and
Neff proved that every objective and isotropic rank-one-convex energy
on \(\mathrm{GL}^+(2)\) is polyconvex \cite{MartinGhibaNeff2017}.

Connections with quasiconformal mapping theory are organized around
the Burkholder integrands.  Astala, Iwaniec, Prause, and Saksman proved
sharp quasiconcavity inequalities for these integrands under a
pointwise distortion restriction \cite{AstalaEtAl2012}.  Guerra and
Kristensen later obtained automatic polyconvexity and conditional
quasiconvexity results for nonnegative homogeneous isotropic
rank-one-convex integrands \cite{GuerraKristensen2022}.  More recently,
Astala, Faraco, Guerra, Koski, and Kristensen proved quasiconvexity of
the local Burkholder functional and obtained quasiconvex energies on
the set of matrices with positive determinant which are not
polyconvex \cite{AstalaEtAl2023}.

Finite elasticity provides another planar candidate family.  For
objective and isotropic energies with an additive
volumetric--isochoric split, Voss, Martin, Ghiba, and Neff reduced part
of Morrey's problem to the quasiconvexity of the rank-one-convex energy
\(W_{\mathrm{magic}}^+\) \cite{VossEtAl2022a}.  Numerical tests based on
finite elements, derivative-free optimization, and neural networks
did not find a violation of quasiconvexity for this energy
\cite{VossEtAl2022b}.

The Dacorogna--Marcellini family was introduced to study this planar
gap.  For \(\gamma\in\R\), it is given by
\begin{equation}\label{eq:def-f}
 f_\gamma(A)=|A|^4-2\gamma|A|^2\det A,
 \qquad A\in\M,
\end{equation}
where \(|A|\) is the Frobenius norm.  Dacorogna and Marcellini
introduced this quartic family in \cite{DacorognaMarcellini1988}.  Its
local convexity conditions can be computed exactly, while part of the
family lies beyond the polyconvex range, which makes it a useful model
for the planar problem.

The analysis of Dacorogna, Douchet, Gangbo, and Rappaz and of Alibert
and Dacorogna \cite{DacorognaEtAl1990,AlibertDacorogna1992} showed that
\[
 \begin{aligned}
  f_\gamma\text{ is convex}
  &\quad\Longleftrightarrow\quad
  |\gamma|\leq\frac{2\sqrt2}{3},\\
  f_\gamma\text{ is polyconvex}
  &\quad\Longleftrightarrow\quad
  |\gamma|\leq1,\\
  f_\gamma\text{ is rank-one convex}
  &\quad\Longleftrightarrow\quad
 |\gamma|\leq\frac{2}{\sqrt3}.
 \end{aligned}
\]
The endpoint \(|\gamma|=2/\sqrt3\) is rank-one convex but not
polyconvex, while polyconvexity proves quasiconvexity only for
\(|\gamma|\leq1\).  This leaves the interval
\[
 1<|\gamma|\leq\frac{2}{\sqrt3}
\]
open.  Failure of quasiconvexity in this interval would provide a
counterexample to the planar implication, while validity up to the
endpoint would make the quasiconvexity and rank-one-convexity
thresholds coincide for the whole family.

The first investigations were numerical, and Dacorogna, Douchet,
Gangbo, and Rappaz found evidence for quasiconvexity at the endpoint
\cite{DacorognaEtAl1990}.  Alibert and Dacorogna then proved that there
exists
\(\varepsilon>0\) such that
\[
 |\gamma|\leq1+\varepsilon
 \quad\Longrightarrow\quad
 f_\gamma\text{ is quasiconvex},
\]
and hence obtained an explicit quasiconvex integrand which is not
polyconvex \cite{AlibertDacorogna1992}.  Their argument crossed the
polyconvex threshold but did not reach the full rank-one-convexity
range, so the remaining gap continued to be studied numerically.
Gremaud used a simulated-annealing method, while Dacorogna and
Haeberly compared several numerical schemes, and both studies found
evidence that quasiconvexity persists throughout the rank-one-convex
range \cite{Gremaud1995,DacorognaHaeberly1998}.  Guti\'errez and
Villavicencio studied a broader optimization problem for quartic
polynomials without finding a planar counterexample
\cite{GutierrezVillavicencio2007}.  Dong and Enakoutsa later reported
numerical violations for the Dacorogna--Marcellini family
\cite{DongEnakoutsa2022}.  The different outcomes did not determine
the quasiconvexity threshold.

\begin{theorem}\label{thm:main}
For \(|\gamma|\leq2/\sqrt3\), the function \(f_\gamma\) defined in
\eqref{eq:def-f} is quasiconvex.
\end{theorem}

Since quasiconvexity implies rank-one convexity, the known obstruction
beyond \(2/\sqrt3\) is also necessary.  Thus
\[
 f_\gamma\text{ is quasiconvex}
 \quad\Longleftrightarrow\quad
 f_\gamma\text{ is rank-one convex}
 \quad\Longleftrightarrow\quad
 |\gamma|\leq\frac2{\sqrt3}.
\]
In particular, the whole interval
\(1<|\gamma|\leq2/\sqrt3\) consists of quasiconvex integrands which are
not polyconvex.

Our method also applied to more general invariant quartic homogeneous polynomials.

\begin{theorem}\label{cor:invariant-quartics}
Let \(Q:\M\to\R\) be a homogeneous polynomial of degree 4 such that
\[
 Q(UFV^T)=Q(F)
 \qquad
 \text{for every }F\in\M\text{ and }U,V\in SO(2).
\]
Then \(Q\) is quasiconvex if and only if it is rank-one convex.
\end{theorem}
The proof of Theorem \ref{cor:invariant-quartics} proceeds in the same way as the proof of Theorem \ref{thm:main}, once one observes that every such polynomial admits a unique representation
\[
 Q(F)=a|F|^4+b|F|^2\det F+c(\det F)^2.
\]
for $a,b,c\in\R$.

\paragraph{Proof strategy.}
The proof uses a monotonicity argument rather than a direct estimate of
\eqref{eq:quasiconvexity}.  Given a compactly supported perturbation
\(\varphi\), we evolve it by the heat equation
\[
 \varphi_\tau=e^{\tau\Delta}\varphi.
\]
As \(\tau\to\infty\), the perturbation disappears, and differentiation
of the relative energy along this path reduces the proof to the
integrated second-variation inequality
\[
 \int_{\R^2}\sum_{i=1}^2
 D^2f_\gamma(Dv)[\partial_iDv,\partial_iDv],\dd x\geq0.
\]
This step is related to the flow-interchange method of Matthes, McCann,
and Savar\'e \cite{MatthesMcCannSavare2009}, since the functional is
studied along an auxiliary flow with simple long-time behavior.

Since the second-variation density need not be pointwise nonnegative,
we construct an
\(SO(2)\times SO(2)\)-invariant exact two-form
\(\omega_\gamma=\dd\alpha_\gamma\) on matrix space.  Its pullback by a
gradient has zero integral,
\[
 \int_{\R^2}(Dv)^*\omega_\gamma=0.
\]
  We choose \(\omega_\gamma\) so that adding it to
the second-variation density makes the resulting expression
pointwise nonnegative, thereby changing the density without changing
its integral.

The rotational invariance of both the energy and the correction
reduces the pointwise inequality to matrices of the form
\(\operatorname{diag}(1,\rho)\).  At the endpoint
\(|\gamma|=2/\sqrt3\), it remains to prove that an explicit quadratic
form is nonnegative, while the intermediate values of \(\gamma\) follow
by interpolation.  The next section develops the heat-flow
criterion, and the final section constructs the exact correction and
verifies the pointwise inequality.

\paragraph{Statement on AI use}
We acknowledge the use of AI tools in both devising important ideas for this paper, and in the writing of the manuscript. In particular, the exact two-form used in the proof was found in an autonomous way by an AI assisted method using an agentic system, Gemini / Aletheia2.

\paragraph{Acknowledgements}
The authors acknowledge support from Thang Luong and Garrett Bingham (Google DeepMind) for providing access to the models used in the proof. We also acknowledge Andr\'e Guerra for comments on a preliminary version of the manuscript.

\section{Monotonicity under heat flow}

To study the functional generated by \(f_\gamma\), we use a strategy
related to the flow-interchange technique of Matthes, McCann, and
Savar\'e from Wasserstein gradient flow theory
\cite{MatthesMcCannSavare2009}.  The idea is to compute the evolution of a functional
along an auxiliary flow rather than along its own gradient flow. Here, the auxiliary flow is the heat flow.

For \(A,B_1,B_2\in\M\), define
\begin{equation*}
 \mathcal E_\gamma(A; B_1,B_2)
 =\sum_{i=1}^2D^2f_\gamma(A)[B_i,B_i].
\end{equation*}

\begin{lemma}\label{lem:heat-flow}
Fix \(\gamma\in\R\).  Suppose that
\begin{equation}\label{eq:integrated-hessian}
 \int_{\R^2}\mathcal E_\gamma
 \bigl(Dv;\partial_1Dv,\partial_2Dv\bigr)\,\dd x\geq0
\end{equation}
whenever \(v(x)=Fx+\psi(x)\), with \(F\in\M\) and
\(\psi\) in the Schwartz class.  Then \(f_\gamma\) is quasiconvex.
\end{lemma}

\begin{proof}
Fix \(F\in\M\) and
\(\varphi\in C_c^\infty(\R^2,\R^2)\).  Let
\[
 \varphi_\tau=e^{\tau\Delta}\varphi,
 \qquad
 G_\tau=D\varphi_\tau,
 \qquad
 A_\tau=F+G_\tau.
\]
Define
\begin{equation*}
 \mathcal F(\tau)
 =\int_{\R^2}\bigl[
 f_\gamma(A_\tau)-f_\gamma(F)
 -\operatorname{tr}\bigl(Df_\gamma(F)G_\tau^T\bigr)
 \bigr],\dd x.
\end{equation*}
For \(\tau>0\), integration by parts and
\(\partial_\tau A_\tau=\Delta A_\tau\) give
\begin{equation*}
 \mathcal F'(\tau)
 =-\int_{\R^2}\mathcal E_\gamma
 \bigl(A_\tau;\partial_1A_\tau,\partial_2A_\tau\bigr)\,\dd x.
\end{equation*}
The hypothesis applied to \(v_\tau(x)=Fx+\varphi_\tau(x)\) shows that
\(\mathcal F\) is nonincreasing. The heat kernel estimate in two dimensions gives
\begin{equation}\label{eq:heat-estimate}
 \|G_\tau\|_{L^p}
 \leq C_p\tau^{-(1-1/p)}\|D\varphi\|_{L^1},
 \qquad p=2,3,4.
\end{equation}
Since \(f_\gamma\) is quartic,
\begin{align}
 \bigl|f_\gamma(F+G)-f_\gamma(F)
 -\operatorname{tr}\bigl(Df_\gamma(F)G^T\bigr)\bigr|
 \leq C_\gamma\bigl(
 |F|^2|G|^2+|F||G|^3+|G|^4\bigr).
 \label{eq:quartic-bound}
\end{align}
Equations \eqref{eq:heat-estimate} and \eqref{eq:quartic-bound} imply
that \(\mathcal F(\tau)\to0\) as \(\tau\to\infty\).

To justify the limit at zero, put
\[
 \mathcal R_F(G)
 =f_\gamma(F+G)-f_\gamma(F)
 -\operatorname{tr}\bigl(Df_\gamma(F)G^T\bigr).
\]
Since \(\mathcal R_F\) is the sum of terms of degrees two, three, and
four in \(G\),
\begin{align*}
 |\mathcal R_F(G)-\mathcal R_F(H)|
 \leq C_\gamma\bigl[&
 |F|^2(|G|+|H|)
 +|F|(|G|^2+|H|^2)
 \\
 &+|G|^3+|H|^3\bigr]|G-H|.
\end{align*}
Writing \(G_0=D\varphi\) and applying H\"older's inequality gives
\begin{align}
 |\mathcal F(\tau)-\mathcal F(0)|
 \leq C_\gamma\bigl[&
 |F|^2(\|G_\tau\|_{L^2}+\|G_0\|_{L^2})
       \|G_\tau-G_0\|_{L^2}
 \nonumber\\
 &+|F|(\|G_\tau\|_{L^3}^2+\|G_0\|_{L^3}^2)
       \|G_\tau-G_0\|_{L^3}
 \nonumber\\
 &+(\|G_\tau\|_{L^4}^3+\|G_0\|_{L^4}^3)
       \|G_\tau-G_0\|_{L^4}\bigr].
 \label{eq:relative-energy-continuity}
\end{align}
For \(p=2,3,4\), strong continuity of the heat semigroup gives
\[
 \|G_\tau-G_0\|_{L^p}
 =\|e^{\tau\Delta}G_0-G_0\|_{L^p}\longrightarrow0
 \qquad(\tau\to0).
\]
The remaining norms in \eqref{eq:relative-energy-continuity} are
bounded for small \(\tau\).  Hence
\(\mathcal F(\tau)\to\mathcal F(0)\) as \(\tau\to0\).
Monotonicity therefore implies \(\mathcal F(0)\geq0\).  Since
\(\int_{\R^2}D\varphi\,\dd x=0\), this is
\[
 \int_{\R^2}\bigl[f_\gamma(F+D\varphi)-f_\gamma(F)\bigr]\dd x
 \geq0.
\]
For \(\varphi\in W^{1,\infty}_0(\Omega,\R^2)\), choose
\(\varphi_j\in C_c^\infty(\Omega,\R^2)\) converging to \(\varphi\)
in \(W^{1,4}\).  The estimate
\[
 |f_\gamma(F+G)-f_\gamma(F+H)|
 \leq C_\gamma(|F|+|G|+|H|)^3|G-H|
\]
and H\"older's inequality allow passage to the limit.  This proves
\eqref{eq:quasiconvexity}.
\end{proof}

\section{The correction lemma}

Let \(\alpha\) be a polynomial one-form on \(\M\), and let
\(\omega=\dd\alpha\).  If \(v(x)=Fx+\psi(x)\), where \(\psi\) is in
the Schwartz class, then
\[
 (Dv)^*\omega=\dd\bigl((Dv)^*\alpha\bigr).
\]
Every coefficient of \((Dv)^*\alpha\) contains a derivative of
\(Dv\), so this one-form is rapidly decreasing.  Stokes' theorem gives
\begin{equation}\label{eq:exact-integral-zero}
 \int_{\R^2}\omega_{Dv}
 \bigl(\partial_1Dv,\partial_2Dv\bigr)\,\dd x=0.
\end{equation}
Fix \(|\gamma|\leq2/\sqrt3\).  To prove the sufficiency assertion in
Theorem~\ref{thm:main}, Lemma~\ref{lem:heat-flow} shows that it is
enough to prove \eqref{eq:integrated-hessian}.  In general, its
integrand
$
 \mathcal E_\gamma(Dv;\partial_1Dv,\partial_2Dv)
$
is not pointwise nonnegative.  We therefore seek an exact two-form
\(\omega_\gamma\) whose addition makes this expression pointwise
nonnegative.  If this form has a polynomial primitive, its integral
vanishes by \eqref{eq:exact-integral-zero}.  Integrating the corrected
inequality then gives \eqref{eq:integrated-hessian}, and
Lemma~\ref{lem:heat-flow} proves quasiconvexity.

\begin{lemma}\label{lem:pointwise}
If \(|\gamma|\leq2/\sqrt3\), there is an exact two-form
\(\omega_\gamma\) on \(\M\), with homogeneous quadratic coefficients
and invariant under
\[
 A\longmapsto OAU^T,
 \qquad O,U\in SO(2),
\]
such that
\begin{equation*}
 \mathcal E_\gamma(Dv;\partial_1Dv,\partial_2Dv)
 +(\omega_\gamma)_{Dv}(\partial_1Dv,\partial_2Dv)\geq0
\end{equation*}
for every \(v\in C^2(\Omega,\R^2)\).
\end{lemma}

\begin{proof} It is enough to check this property for $\gamma = 2/\sqrt{3}$, and we are going to reduce to this case throughout the proof.

\vspace{10pt}
 
\noindent {\bf Step 1.} The function \(f_\gamma\) is invariant under rotations on the
left and on the right.  We impose the same invariance on the correcting
form.  Since the correction must be exact, we write
\(\omega=\dd\alpha\) and seek an invariant one-form \(\alpha\).
The coefficients of \(\mathcal E_\gamma\) are homogeneous of degree two
in its first matrix argument.  Thus \(\omega\) must have homogeneous
quadratic coefficients, and \(\alpha\) must have homogeneous cubic
coefficients.

Write \(A=(\mathbf a_1\mid\mathbf a_2)\) and
\(B=(\mathbf b_1\mid\mathbf b_2)\).  For vectors in \(\R^2\), set
\(\mathbf a\times\mathbf b=a_1b_2-a_2b_1\).  Two invariant linear
one-forms are
\begin{align}
 \theta_A(B)
 &=\mathbf a_1\times\mathbf b_1
   +\mathbf a_2\times\mathbf b_2,
 \label{eq:theta-cross}\\
 \zeta_A(B)
 &=\mathbf a_1\cdot\mathbf b_2
   -\mathbf a_2\cdot\mathbf b_1.
 \label{eq:theta-dot}
\end{align}
Let
$
 J=\begin{pmatrix}0&-1\\1&0\end{pmatrix}.
$
Then
\[
 \theta_A(B)=\operatorname{tr}(JA B^T),
 \qquad
 \zeta_A(B)=\operatorname{tr}(AJ^T B^T).
\]
The one-forms in \eqref{eq:theta-cross}--\eqref{eq:theta-dot} have linear coefficients in \(A\).  To make \(\alpha\) homogeneous of degree three, we
multiply them by the quadratic scalar quantities already present in the
definition of \(f_\gamma\), namely \(|A|^2\) and \(\det A\).  We use
the four-parameter ansatz
\begin{equation}\label{eq:ansatz}
 \alpha
 =c_1|A|^2\theta
  +c_2(\det A)\zeta
  +c_3|A|^2\zeta
  +c_4(\det A)\theta,
 \qquad
 \omega=\dd\alpha,
\end{equation}
where \(c_1,c_2,c_3,c_4\in\R\).

In terms of the entries of \(A\),
\begin{align}
 \theta
 &=A_{11}\,\dd A_{21}-A_{21}\,\dd A_{11}
   +A_{12}\,\dd A_{22}-A_{22}\,\dd A_{12},
 \label{eq:theta-cross-entries}\\
 \zeta
 &=A_{11}\,\dd A_{12}-A_{12}\,\dd A_{11}
   +A_{21}\,\dd A_{22}-A_{22}\,\dd A_{21}.
 \label{eq:theta-dot-entries}
\end{align}
Every rotation commutes with \(J\), since $J$ itself is a rotation.  If \(T(A)=OAU^T\), with
\(O,U\in SO(2)\), then
\begin{align*}
 (T^*\theta)_A(B)
 &=\operatorname{tr}\bigl(JOAU^T(OBU^T)^T\bigr)
 =\operatorname{tr}(JA B^T),\\
 (T^*\zeta)_A(B)
 &=\operatorname{tr}\bigl(OAU^TJ^T(OBU^T)^T\bigr)
 =\operatorname{tr}(AJ^T B^T).
\end{align*}
The functions \(|A|^2\) and \(\det A\) are invariant as well.
Consequently, \(\alpha\) and \(\omega\) in \eqref{eq:ansatz} are
invariant under left and right multiplication.  The form \(\omega\) is exact and has homogeneous quadratic coefficients.
\vspace{10pt}

\noindent
{\bf Step 2.} For \(A,B\in\M\),
\begin{align*}
 |A+sB|^2
 &=|A|^2+2s\operatorname{tr}(AB^T)+s^2|B|^2,\\
 \det(A+sB)
 &=\det A+s\operatorname{tr}\bigl((\cof A)B^T\bigr)
   +s^2\det B.
\end{align*}
Differentiating \eqref{eq:def-f} twice gives
\begin{align}
 D^2f_\gamma(A)[B,B]
 ={}&8\operatorname{tr}(AB^T)^2
 +4\bigl(|A|^2-\gamma\det A\bigr)|B|^2
 \nonumber\\
 &-8\gamma\operatorname{tr}(AB^T)
       \operatorname{tr}\bigl((\cof A)B^T\bigr)
 -4\gamma |A|^2\det B.
 \label{eq:hessian}
\end{align}
\vspace{10pt}

\noindent
{\bf Step 3.}
 Fix a point \(x_0 \in \mathbb{R}^2\) and constant matrices \(O,U\in SO(2)\).  Define
\[
 \widetilde v(y)=O\bigl(v(x_0+U^Ty)-v(x_0)\bigr).
\]
If \(A=Dv(x_0)\) and \(B_i=\partial_iDv(x_0)\), then
\begin{align*}
 \widetilde A:=D\widetilde v(0)&=OAU^T,
 \\
 \widetilde B_i:=\partial_iD\widetilde v(0)
 &=\sum_{k=1}^2U_{ik}\,OB_kU^T.
\end{align*}
The invariance of \(f_\gamma\), the orthogonality of \(U\), and the
skwe-symmetry of \(\omega\) give
\begin{align*}
 \mathcal E_\gamma(\widetilde A;\widetilde B_1,\widetilde B_2)
 &=\mathcal E_\gamma(A;B_1,B_2),
 \\
 \omega_{\widetilde A}
 (\widetilde B_1,\widetilde B_2)
 &=\omega_A(B_1,B_2).
\end{align*}
By singular value decomposition, there exist \(O,U\in SO(2)\) such that
$OAU^T=\diag(\sigma,\tau)$, with $
 \sigma\geq|\tau|$. It is therefore enough to restrict our reasoning to $A$ having the form $ A_\rho=\diag(1,\rho)$,
with $\rho \in \mathbb{R}$.

\vspace{10pt}

\noindent
{\bf Step 4.} 
At  $x_0$, we may write
\begin{equation*}
 B_1=\begin{pmatrix}\xi_1&\xi_2\\\eta_1&\eta_2\end{pmatrix},
 \qquad
 B_2=\begin{pmatrix}\xi_2&\xi_3\\\eta_2&\eta_3\end{pmatrix}.
\end{equation*}
At \(\gamma=2/\sqrt3\), substitution in \eqref{eq:hessian} gives
\begin{equation*}
 \mathcal E_{2/\sqrt3}(A_\rho;B_1,B_2)=Q_1+Q_2,
\end{equation*}
where $Q_1$ and $Q_2$ are given by
\begin{align}
 Q_1(\rho;\xi_1,\xi_3,\eta_2)={}&
 4(\rho-\sqrt3)^2\xi_1^2
 +4\left(1+\rho^2-\frac{2\rho}{\sqrt3}\right)\xi_3^2
 \nonumber\\
 &+4\left(2+4\rho^2-\frac{8\rho}{\sqrt3}\right)\eta_2^2
 +4\bigl(4\rho-2\sqrt3(1+\rho^2)\bigr)\xi_1\eta_2
 \nonumber\\
 &+\frac{8}{\sqrt3}(1+\rho^2)\xi_3\eta_2,
 \label{eq:Q1}\\
 Q_2(\rho;\eta_3,\eta_1,\xi_2)={}&
 4\left(1+\rho^2-\frac{2\rho}{\sqrt3}\right)\eta_1^2
 +4\left(4+2\rho^2-\frac{8\rho}{\sqrt3}\right)\xi_2^2
 \nonumber\\
 &+4(1+3\rho^2-2\sqrt3\rho)\eta_3^2
 +\frac{8}{\sqrt3}(1+\rho^2)\eta_1\xi_2
 \nonumber\\
 &+4\bigl(4\rho-2\sqrt3(1+\rho^2)\bigr)\xi_2\eta_3.
 \nonumber
\end{align}
Equations \eqref{eq:theta-cross-entries} and
\eqref{eq:theta-dot-entries} give
\begin{align*}
 \dd\theta
 &=2(\dd A_{11}\wedge\dd A_{21}
      +\dd A_{12}\wedge\dd A_{22}),
 \\
 \dd\zeta
 &=2(\dd A_{11}\wedge\dd A_{12}
      +\dd A_{21}\wedge\dd A_{22}).
\end{align*}
At \(A_\rho\),
\[
\begin{array}{c|cc}
 &B_1&B_2\\ \hline
 \dd|A|^2&2(\xi_1+\rho\eta_2)&2(\xi_2+\rho\eta_3)\\
 \dd\det A&\rho\xi_1+\eta_2&\rho\xi_2+\eta_3\\
 \theta&\eta_1-\rho\xi_2&\eta_2-\rho\xi_3\\
 \zeta&\xi_2-\rho\eta_1&\xi_3-\rho\eta_2
\end{array}
\]
Using \(\dd(g\vartheta)=\dd g\wedge\vartheta
+g\,\dd\vartheta\), the correction splits as
\begin{equation*}
 \omega_{A_\rho}(B_1,B_2)
 =C_1(\rho;\xi_1,\xi_3,\eta_2)
  +C_2(\rho;\eta_3,\eta_1,\xi_2),
\end{equation*}
where
\begin{align}
 C_1={}&
 \bigl[-2\rho c_1+3\rho c_2+(4+2\rho^2)c_3-\rho^2c_4\bigr]
 \xi_1\xi_3
 \nonumber\\
 &+\bigl[(4+2\rho^2)c_1-\rho^2c_2-2\rho c_3+3\rho c_4\bigr]
 \xi_1\eta_2
 \nonumber\\
 &+\bigl[-(2+4\rho^2)c_1+c_2+2\rho c_3-3\rho c_4\bigr]
 \xi_3\eta_2
 \nonumber\\
 &+\bigl[2\rho c_1-3\rho c_2-(2+4\rho^2)c_3+c_4\bigr]
 \eta_2^2.
 \label{eq:general-C1}
\end{align}
The second block is
\begin{align*}
 C_2={}&
 \bigl[-2\rho c_1+3\rho c_2+(2+4\rho^2)c_3-c_4\bigr]
 \eta_1\eta_3
 \\
 &+\bigl[-(4+2\rho^2)c_1+\rho^2c_2+2\rho c_3-3\rho c_4\bigr]
 \eta_1\xi_2
 \\
 &+\bigl[(2+4\rho^2)c_1-c_2-2\rho c_3+3\rho c_4\bigr]
 \eta_3\xi_2
 \\
 &+\bigl[2\rho c_1-3\rho c_2-(4+2\rho^2)c_3+\rho^2c_4\bigr]
 \xi_2^2.
\end{align*}
The \(\xi_3^2\)-coefficient in \(Q_1\) is
$4\kappa(\rho)$, and $
 \kappa(\rho)=1+\rho^2-\frac{2\rho}{\sqrt3}
 =\left(\rho-\frac1{\sqrt3}\right)^2+\frac23>0$.
We impose that the corrected \(\eta_2^2\)-coefficient be equal
\(4\kappa(\rho)\) (this choice makes the eigenvectors of the $\eta_2$-$\xi_3$ minor independent of $\rho$). Matching the powers of \(\rho\) in
\eqref{eq:general-C1} gives
$c_3=3, c_4=2$, and $2c_1-3c_2=8\sqrt3$.
Therefore, the \(\xi_1\xi_3\)-coefficient in
\(Q_1+C_1\) is \(4(\rho-\sqrt3)^2\).
At
\(\rho=\sqrt3\), the \(\xi_1^2\)-coefficient vanishes.  A
nonnegative quadratic form with a zero \(\xi_1^2\)-coefficient cannot
have a nonzero mixed term involving \(\xi_1\).  The
\(\xi_1\eta_2\)-coefficient at \(\rho=\sqrt3\) is
\(8(c_1-\sqrt3)\).  Hence we set
\begin{equation*}
 c_1=\sqrt3,
 \qquad c_2=-2\sqrt3,
 \qquad c_3=3,
 \qquad c_4=2.
\end{equation*}
The corresponding primitive and two-form are
\begin{align*}
 \alpha_+
 &=(\sqrt3|A|^2+2\det A)\theta
   +(3|A|^2-2\sqrt3\det A)\zeta,
 \\
 \omega_+&=\dd\alpha_+.
\end{align*}
For these coefficients, \eqref{eq:general-C1} becomes
\begin{align}
 C_1={}&4(\rho-\sqrt3)^2\xi_1\xi_3
 +4\sqrt3(1+\rho^2)(\xi_1\eta_2-\xi_3\eta_2)-12\left(\rho-\frac1{\sqrt3}\right)^2\eta_2^2.
 \label{eq:C1}
\end{align}
Adding \eqref{eq:C1} to \eqref{eq:Q1}, we can write the corrected
block directly as
\begin{equation*}
 Q_1+C_1
 =4
 \begin{pmatrix}\xi_1&\xi_3&\eta_2\end{pmatrix}
 M_1(\rho)
 \begin{pmatrix}\xi_1\\\xi_3\\\eta_2\end{pmatrix},
\end{equation*}
where
\begin{equation*}
 M_1(\rho)=
 \begin{pmatrix}
 (\rho-\sqrt3)^2
 &\frac12(\rho-\sqrt3)^2
 &2\rho-\frac{\sqrt3}{2}(1+\rho^2)\\[2pt]
 \frac12(\rho-\sqrt3)^2
 &1+\rho^2-\frac{2\rho}{\sqrt3}
 &-\frac{1+\rho^2}{2\sqrt3}\\[2pt]
 2\rho-\frac{\sqrt3}{2}(1+\rho^2)
 &-\frac{1+\rho^2}{2\sqrt3}
 &1+\rho^2-\frac{2\rho}{\sqrt3}
 \end{pmatrix}.
\end{equation*}
The principal minors of \(M_1(\rho)\) are
\begin{align*}
 \Delta_1
 &=(\rho-\sqrt3)^2,
 \\
 \Delta_2
 &=\frac34(\rho-\sqrt3)^2
 \left[\left(\rho-\frac{\sqrt3}{9}\right)^2
       +\frac8{27}\right],
 \\
 \Delta_3
 &=\det M_1(\rho)
 =\frac16(\rho-\sqrt3)^2(1-\rho^2)^2.
\end{align*}
If \(\rho\notin\{-1,1,\sqrt3\}\), all three are strictly positive,
so \(M_1(\rho)>0\).  Since
\(M_1(\rho)\) depends continuously on \(\rho\) and the set of
positive semidefinite matrices is closed, it follows at the three
critical values as well that \(M_1(\rho)\geq0\).  Therefore
\begin{equation*}
 Q_1+C_1\geq0
 \qquad\text{for every }\rho,\xi_1,\xi_3,\eta_2\in\R.
\end{equation*}
For the second block, $Q_2+C_2$, we do not need to repeat the previous calculations. Indeed, the transformation in Step~3 with \(O=U=J\)
is
\[
 \begin{gathered}
 A_\rho\longmapsto \rho A_{1/\rho},\qquad \rho\neq0,\\
 (\xi_1,\xi_3,\eta_2)\longmapsto
   (-\eta_3,-\eta_1,-\xi_2),\\
 (\eta_3,\eta_1,\xi_2)\longmapsto
   (\xi_1,\xi_3,\eta_2).
 \end{gathered}
\]
Thus, setting \(\xi_1=\xi_3=\eta_2=0\), invariance and degree-two
homogeneity in \(A\) give
\begin{align*}
 &Q_2(\rho;\eta_3,\eta_1,\xi_2)
   +C_2(\rho;\eta_3,\eta_1,\xi_2)\\
 &\qquad=\rho^2\bigl[
 Q_1(1/\rho;\eta_3,\eta_1,\xi_2)
 +C_1(1/\rho;\eta_3,\eta_1,\xi_2)\bigr]\geq0,
 \qquad \rho\neq0.
\end{align*}
The three minus signs have disappeared because all terms are
quadratic monomials.  At \(\rho=0\) the non-negativity follows by continuity.

\end{proof}

\appendix

\section{The invariant quartic cone}\label{app:invariant-quartics}

In this appendix we briefly record how the argument for the
Dacorogna--Marcellini family extends to prove
Theorem~\ref{cor:invariant-quartics}.  Set
\[
 s(F)=|F|^2,\qquad d(F)=\det F.
\]

\begin{proof}[Proof of Theorem~\ref{cor:invariant-quartics}]
\medskip\noindent\emph{Invariant representation.}
We first show that every homogeneous quartic polynomial satisfying the
invariance assumption in Theorem~\ref{cor:invariant-quartics} has a
unique representation of the form
\begin{equation}\label{eq:app-general-quartic}
 Q(F)=a s(F)^2+b s(F)d(F)+c d(F)^2.
\end{equation}
Let
\[
 q(x,y)=Q(\diag(x,y)).
\]
Since \(Q\) is homogeneous of degree four, \(q\) is a homogeneous
quartic polynomial in \(x\) and \(y\).  Moreover, with
\[
 J=\begin{pmatrix}0&-1\\1&0\end{pmatrix},
 \qquad
 J\diag(x,y)J^T=\diag(y,x),
\]
and hence the rotational invariance of \(Q\) gives
\(q(x,y)=q(y,x)\).  Every symmetric homogeneous quartic in two
variables has the form
\[
 q(x,y)=A(x^4+y^4)+B(x^3y+xy^3)+Cx^2y^2.
\]
On diagonal matrices, \(s=x^2+y^2\) and \(d=xy\), so
\[
 x^4+y^4=s^2-2d^2,\qquad
 x^3y+xy^3=sd,\qquad x^2y^2=d^2.
\]
Thus \eqref{eq:app-general-quartic} holds on every diagonal matrix.
Every matrix \(F\in\M\) can be brought to a signed diagonal matrix by
proper rotations on the left and right.  Since both sides of
\eqref{eq:app-general-quartic} are invariant under these rotations,
the identity holds for every \(F\).  Finally, \(s^2,sd,d^2\) are
linearly independent already on diagonal matrices, which proves
uniqueness.

\medskip\noindent\emph{Rank-one convexity.}
We next determine rank-one convexity.  By the rotational invariance it
is enough to test the direction
\[
 E_{11}=\begin{pmatrix}1&0\\0&0\end{pmatrix}.
\]
For
\(A=\left(\begin{smallmatrix}x&y\\ z&w\end{smallmatrix}\right)\), direct
differentiation gives
\[
 D^2Q(A)[E_{11},E_{11}]
 =12ax^2+6bxw+(4a+2c)w^2
   +4a(y^2+z^2)-2byz.
\]
Thus this expression is nonnegative for every \(A\) if and only if
\[
 \begin{pmatrix}12a&3b\\3b&4a+2c\end{pmatrix}\geq0,
 \qquad
 \begin{pmatrix}4a&-b\\-b&4a\end{pmatrix}\geq0.
\]
Equivalently, either
\begin{equation}\label{eq:app-quartic-cone}
 a>0,\qquad |b|\leq4a,\qquad
 c\geq-2a+\frac{3b^2}{8a},
\end{equation}
or \(a=b=0\) and \(c\geq0\).  By the Legendre--Hadamard
characterization \eqref{eq:legendre-hadamard}, these conditions
describe the full rank-one convex cone.

\medskip\noindent\emph{Reduction to the boundary family.}
It remains to prove sufficiency for quasiconvexity.  Suppose first that
\(a>0\), and set
\[
 t=\frac{b}{4a},\qquad
 \mu=c+2a-\frac{3b^2}{8a}.
\]
Condition \eqref{eq:app-quartic-cone} says that \(|t|\leq1\) and
\(\mu\geq0\), and gives the decomposition
\begin{equation}\label{eq:app-cone-decomposition}
 Q=aW_t+\mu d^2,
 \qquad
 W_t=s^2+4t sd+(6t^2-2)d^2.
\end{equation}
The term \(d^2\) is polyconvex, hence quasiconvex.  We only have to
consider \(W_t\).

\medskip\noindent\emph{Quasiconvexity of the boundary family.}
We use the invariant one-forms \(\theta\) and \(\zeta\) defined in
\eqref{eq:theta-cross}--\eqref{eq:theta-dot}.  Define the cubic
one-form and its exact differential by
\begin{equation}\label{eq:app-alpha-t}
 \begin{aligned}
 \alpha_t
 &=[-3t s+6(1-2t^2)d]\theta
   +[3s+6td]\zeta,\\
 \omega_t&=\dd\alpha_t.
 \end{aligned}
\end{equation}
Both \(\alpha_t\) and \(\omega_t\) are invariant under proper rotations
on the left and right.

By the same argument used in Step~3 of the proof of
Lemma~\ref{lem:pointwise}, it suffices to take
\(A=A_\rho=\diag(1,\rho)\).  At the point under consideration,
write the derivative matrices as in the main proof:
\[
 B_1=\begin{pmatrix}\xi_1&\xi_2\\\eta_1&\eta_2\end{pmatrix},
 \qquad
 B_2=\begin{pmatrix}\xi_2&\xi_3\\\eta_2&\eta_3\end{pmatrix}.
\]
Put
\[
 L=1+t\rho,\qquad M=t+\rho,\qquad \delta=1-t^2.
\]
Differentiating \(W_t\) and \eqref{eq:app-alpha-t}, and collecting
squares, gives
\begin{align*}
 &D^2W_t(A_\rho)[B_1,B_1]+D^2W_t(A_\rho)[B_2,B_2]
   +(\omega_t)_{A_\rho}(B_1,B_2)\\
 &={}
 3(2L\xi_1+L\xi_3+M\eta_2)^2+(L\xi_3-M\eta_2)^2
   +4\delta(\rho\xi_3-\eta_2)^2\\
 &\quad
 +3(L\xi_2+M\eta_1+2M\eta_3)^2+(L\xi_2-M\eta_1)^2
   +4\delta(\rho\xi_2-\eta_1)^2.
\end{align*}
This is nonnegative because \(|t|\leq1\).  The same covariance and
homogeneity argument gives the corresponding inequality for every
\(A\in\M\).

Now let \(v(x)=Fx+\psi(x)\), with \(\psi\) in the Schwartz class.
Exactness and \eqref{eq:exact-integral-zero} imply
\[
 \int_{\R^2}(\omega_t)_{Dv}
   (\partial_1Dv,\partial_2Dv)\,\dd x=0.
\]
After integration, the pointwise inequality therefore yields
\[
 \int_{\R^2}\sum_{j=1}^2
 D^2W_t(Dv)[\partial_jDv,\partial_jDv],\dd x\geq0.
\]
The proof of Lemma~\ref{lem:heat-flow} uses only that the integrand is
a homogeneous quartic, so it applies verbatim and shows that \(W_t\)
is quasiconvex.  Hence \eqref{eq:app-cone-decomposition} is
quasiconvex.  In the remaining case \(a=b=0\), one has
\(Q=c d^2\) with \(c\geq0\), which was already covered above.

\end{proof}

\end{document}